\documentclass{amsart}
\usepackage[latin9]{inputenc}
\usepackage{geometry}
\usepackage{verbatim}
\usepackage{amstext}
\usepackage{amsthm}
\usepackage{amssymb}
\usepackage{graphicx}
\usepackage{esint}

\makeatletter
\numberwithin{equation}{section}
\numberwithin{figure}{section}
\theoremstyle{plain}
\newtheorem{thm}{\protect\theoremname}
\theoremstyle{remark}
\newtheorem{rem}[thm]{\protect\remarkname}
\theoremstyle{plain}
\newtheorem{prop}[thm]{\protect\propositionname}

\global\long\def\sign{\operatorname{sign}}

\makeatother

\providecommand{\propositionname}{Proposition}
\providecommand{\remarkname}{Remark}
\providecommand{\theoremname}{Theorem}

\begin{document}
\title{Zeros of Taylor polynomials formed by a three-term recurrence}
\author{Juhoon Chung \and Khang Tran}
\address{Department of Mathematics \\
 California State University\\
 5245 North Backer Avenue M/S PB108 Fresno, CA 93740}
\begin{abstract}
From a sequence $\left\{ a_{n}\right\} _{n=0}^{\infty}$ of real numbers
satisfying a three-term recurrence, we form a sequence of polynomials
$\left\{ P_{m}(z)\right\} _{m=0}^{\infty}$ whose coefficients are
numbers in this sequence. We showed that under explicit conditions,
the zeros of $P_{m}(z)$, $m\gg1$, lie on one side of the circle
whose radius is given by modulus of a zero of the denominator of the
generating function of $\left\{ a_{n}\right\} _{n=0}^{\infty}$. 
\end{abstract}

\maketitle

\section{Introduction}

For $a,b,c\in\mathbb{R}$ and $abc\ne0$, we let $\left\{ a_{n}\right\} _{n=0}^{\infty}$
be a sequence of real numbers satisfying the three-term recurrence
\begin{equation}
ca_{n+2}+ba_{n+1}+aa_{n}=0,n\ge0,\label{eq:a_nrecurrence}
\end{equation}
with the initial values $(a_{0},a_{1})\ne(0,0)$. This sequence includes
many classic sequences in mathematics such as the Fibonacci, Lucas,
and Pell sequences (see \cite{wells,hoggatt}). It is known that we
can obtain a closed formula for $a_{n}$ from the zeros, denoted by
$\alpha$ and $\beta$, of the characteristic polynomial $c+bt+at^{2}$.
In this paper, these zeros play a central role in determining the
location of the zeros of Taylor polynomials associated with this sequence.
The sequence of Taylor polynomials $\left\{ P_{m}(z)\right\} _{m=0}^{\infty}$
associated with the sequence $\left\{ a_{n}\right\} _{n=0}^{\infty}$
is defined by 
\[
P_{m}(z)=\sum_{n=0}^{m}a_{n}z^{n}.
\]
The study of locations of zeros of Taylor polynomials from general
sequences (not necessarily from those satisfying \eqref{eq:a_nrecurrence})
is of interest to many mathematicians. A classic example is the case
$a_{n}=1/n!$, where the zeros of the associate Taylor polynomials
approach the famous Szego curve (see \cite{pv}). It is not difficult
to show when $\left\{ a_{n}\right\} _{n=0}^{\infty}$ satisfies \eqref{eq:a_nrecurrence}
and $|\alpha|>|\beta|$, most of the zeros of $P_{m}(z)$ approach
the circle radius $|c|/|a\alpha|=|\beta|$ (c.f. Remark \ref{rem:limitzeros})
as $m\rightarrow\infty$. In this paper we provide conditions under
which the zeros $P_{m}(z),$$m\gg1$, lie on one side of this circle.
There are many studies of polynomials whose zeros lie on one side
of a circle or a half plane. We refer the readers to a development
by Borcea and Branden (\cite{bb}) on linear operator preserving zeros
of polynomials on a circular domain. We state the three main theorems
of this paper. 
\begin{thm}
\label{thm:maintheorem1} Let $\alpha$ be the largest (in modulus)
zero of $at^{2}+bt+c$. If $ac<0$, $a_{0}b(a_{1}c+ba_{0})\ge0$,
then with the possible one exceptional zero, all the zeros of $P_{m}(z)$,
$m\gg1$, lie outside the open ball centered at 0 with radius $|c|/|a\alpha|$.
Moreover, the exceptional zero occurs if and only if 
\[
|a_{1}c+ba_{0}|>|a\alpha a_{0}|.
\]
\end{thm}

For examples of Theorem \ref{thm:maintheorem1}, see Figure \ref{fig:thm1}.

\begin{figure}
\begin{centering}
\includegraphics[scale=0.5]{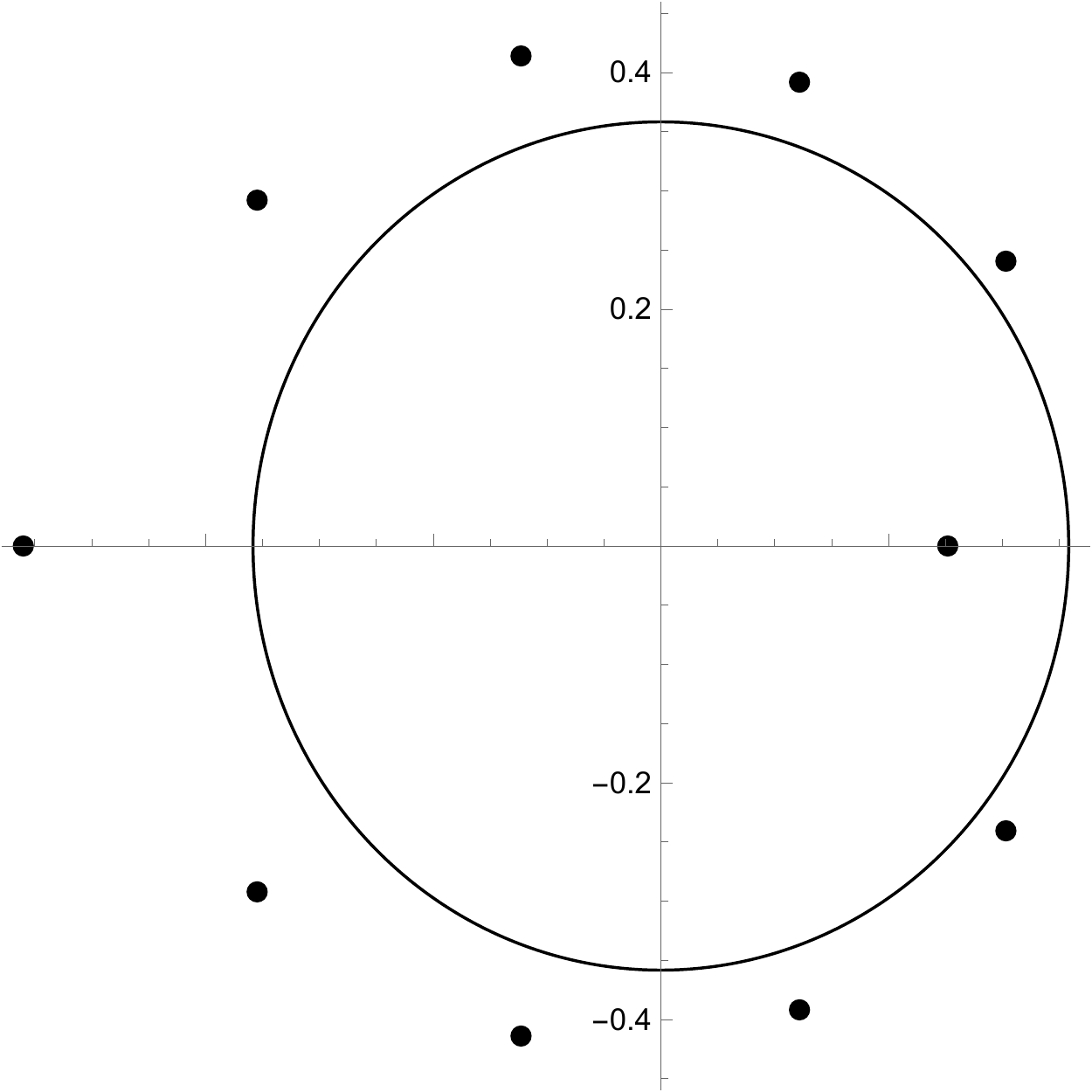}\qquad{}\includegraphics[scale=0.5]{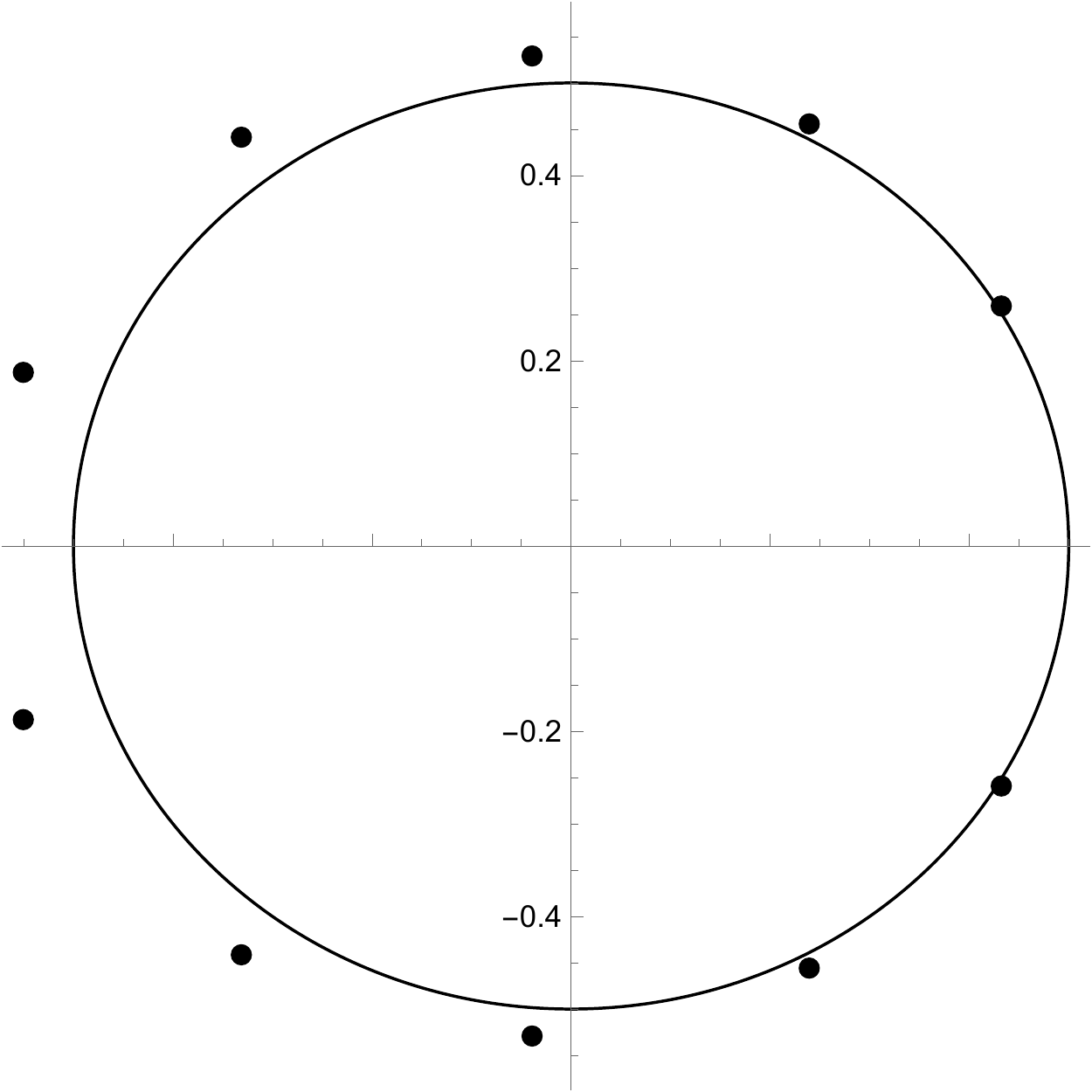} 
\par\end{centering}
\caption{\label{fig:thm1}Zeros of $H_{10}(z)$ when $(a,b,c,a_{0},a_{1})$
is $(5,1,-1,1,-3)$ (left) and $(2,1,-1,2,1)$ (right) }
\end{figure}

\begin{thm}
\label{thm:maintheorem2} Let $\alpha$ be the largest (in modulus)
zero of $at^{2}+bt+c$. If $ac>0$, $b^{2}-4ac>0$, and $a_{0}b(a_{1}c+ba_{0})\le0$,
then for large $m$, all the zeros of $P_{m}(z)$ lie inside the closed
ball centered at 0 with radius $|c|/|a\alpha|$. 
\end{thm}

\begin{figure}
\begin{centering}
\includegraphics[scale=0.5]{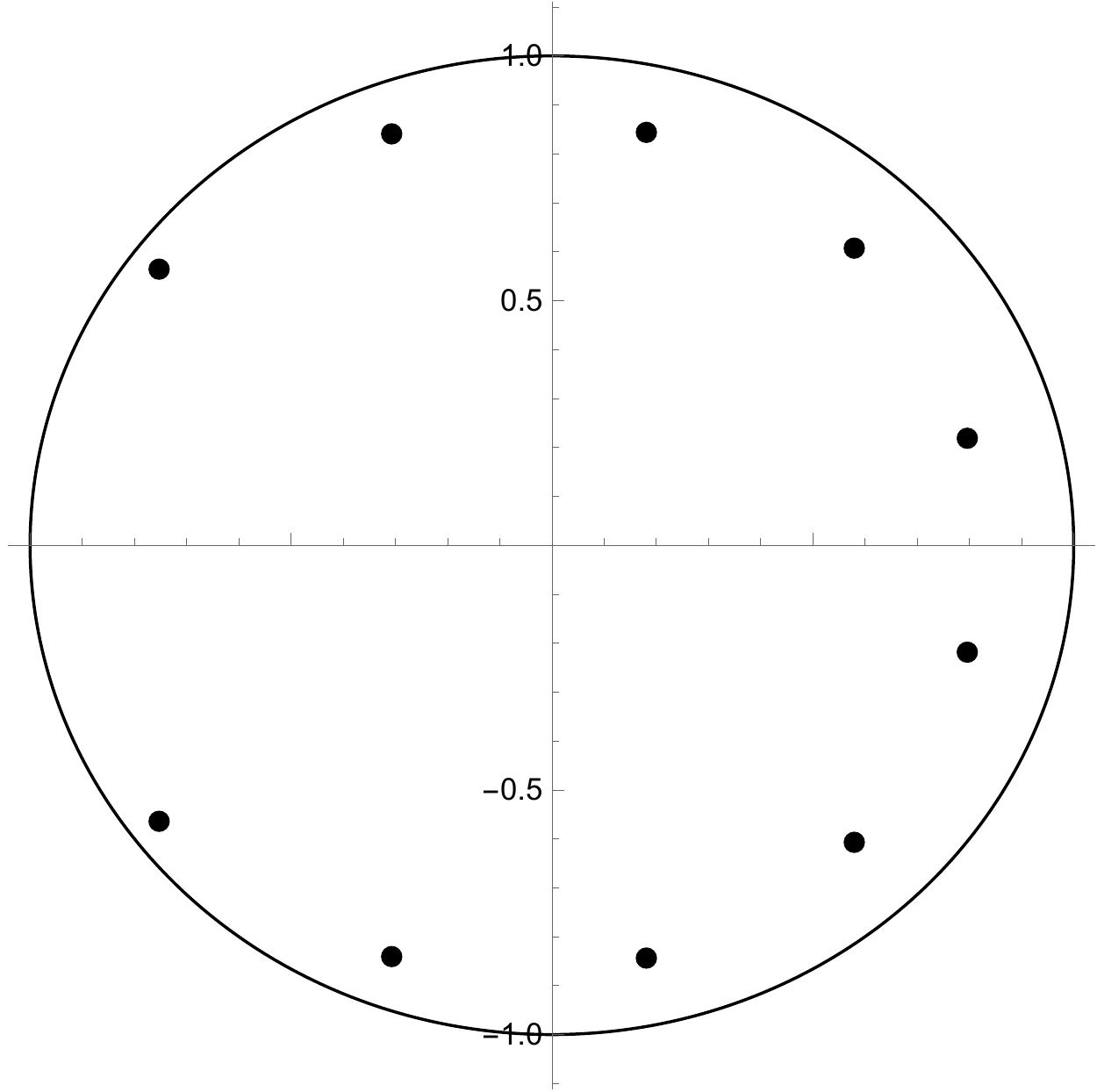}\qquad{}\includegraphics[scale=0.5]{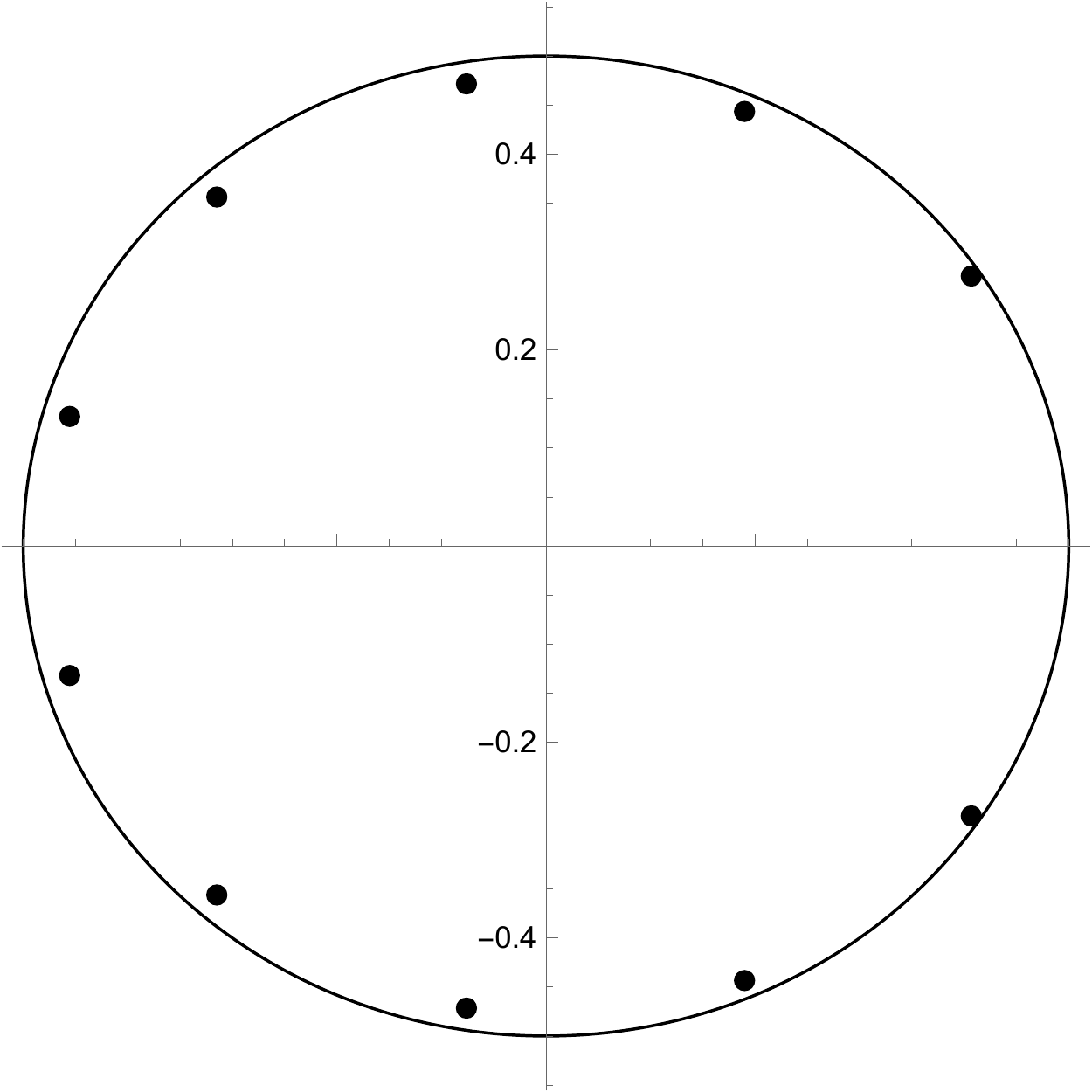} 
\par\end{centering}
\caption{\label{fig:thm23}Zeros of $H_{10}(z)$ when $(a,b,c,a_{0},a_{1})$
is $(2,5,3,1,-2)$ (left) and $(2,-3,1,2,5)$ (right)}
\end{figure}

\begin{thm}
\label{thm:maintheorem3} Let $\alpha$ be the largest (in modulus)
zero of $at^{2}+bt+c$. If $ac>0$, $b^{2}-4ac>0$, $a_{0}b(a_{1}c+ba_{0})>0$,
and $|a_{0}c|>|(a_{1}c+ba_{0})\alpha|$, then for large $m$, all
the zeros of $P_{m}(z)$ lie inside the closed ball centered at 0
with radius $|c|/|a\alpha|$. 
\end{thm}

For examples of Theorem \ref{thm:maintheorem2} and \ref{thm:maintheorem3},
see Figure \ref{fig:thm23} (left and right respectively).

\section{The generating function and its reduced form}

Since the sequence $\left\{ a_{n}\right\} _{n=0}^{\infty}$ satisfies
the recurrence \eqref{eq:a_nrecurrence}, for small $t$ we collect
the $t^{n}$- coefficents of 
\[
(c+bt+at^{2})\sum_{n=0}^{\infty}a_{n}t^{n}
\]
and conclude that this expression is 
\[
ca_{0}+(ca_{1}+ba_{0})t.
\]
Thus the generating function for the sequence $\left\{ a_{n}\right\} _{n=0}^{\infty}$
is 
\[
\sum_{n=0}^{\infty}a_{n}t^{n}=\frac{ca_{0}+(ca_{1}+ba_{0})t}{c+bt+at^{2}}.
\]
To study the zero distribution of the sequence $\left\{ P_{m}(z)\right\} _{m=0}^{\infty}$,
it suffices to study that of the sequence of reciprocal polynomials
\[
P_{m}^{*}(z):=\sum_{n=0}^{m}a_{n}z^{m-n}.
\]
The generating function for this sequence of reciprocal polynomials
is 
\begin{align*}
\sum_{m=0}^{\infty}P_{m}^{*}(z)t^{m} & =\sum_{m=0}^{\infty}\sum_{n=0}^{m}a_{n}z^{m-n}t^{m}\\
 & =\sum_{n=0}^{\infty}\sum_{m=n}^{\infty}a_{n}z^{m-n}t^{m}\\
 & =\sum_{n=0}^{\infty}a_{n}t^{n}\sum_{m=0}^{\infty}z^{m}t^{m}\\
 & =\frac{ca_{0}+(ca_{1}+ba_{0})t}{(c+bt+at^{2})(1-tz)}.
\end{align*}
We divide the numerator and denominator by $c$ and apply the substitutions
$t\rightarrow-\sign(ab)\sqrt{|c/a|}t$ and $z\rightarrow-\sign(ab)\sqrt{|a/c|}z$
to conclude that for small $t$ 
\[
\sum_{m=0}^{\infty}H_{m}(z)t^{m}=\frac{D+Ct}{(1+Bt+\sign(ac)t^{2})(1-tz)}
\]
where 
\begin{align}
H_{m}(z) & =\sign^{m}(-ab)|c/a|^{m/2}P_{m}^{*}\left(-\sign(ab)\sqrt{|a/c|}z\right),\label{eq:HmPm*}\\
B & =-\sign(ac)|b/\sqrt{ac}|,\nonumber \\
C & =-(a_{1}+ba_{0}/c)\sign(ab)\sqrt{|c/a|},\nonumber \\
D & =a_{0}.\nonumber 
\end{align}
We let $t_{1}$ and $t_{2}$ be the two zeros of $1+Bt+\sign(ac)t^{2}$
where $|t_{2}|\ge|t_{1}|$. If $\alpha$ is the largest (in modulus)
zero of $c+bt+at^{2}$, then from the substitutions above we have
\begin{equation}
\alpha=-\sign(ab)\sqrt{|c/a|}t_{2}.\label{eq:alphat2}
\end{equation}
With 
\[
\frac{D+Ct}{(1+Bt+\sign(ac)t^{2})(1-tz)}=\frac{D+Ct}{\sign(ac)(t-t_{1})(t-t_{2})(1-tz)}
\]
we apply partial fractions (if $t_{1}\ne t_{2}$) to rewrite the expression
above as $\sign(ac)$ times 
\[
\frac{D+Ct_{1}}{(t_{1}-t_{2})(1-zt_{1})(t-t_{1})}+\frac{D+Ct_{2}}{(t_{2}-t_{1})(1-zt_{2})(t-t_{2})}+\frac{z(C+Dz)}{(1-zt_{1})(1-zt_{2})(1-zt)}.
\]
From 
\begin{align*}
\frac{1}{t-t_{k}} & =-\sum_{m=0}^{\infty}\frac{t^{m}}{t_{k}^{m+1}},\qquad k=1,2\\
\frac{1}{1-zt} & =\sum_{m=0}^{\infty}z^{m}t^{m},
\end{align*}
we expand the terms above as a power series in $t$ and collect the
$t^{m}$-coefficient to conclude from $t_{1}t_{2}=\sign(ac)$ that
if $t_{1}\ne t_{2}$ , then 
\begin{align}
\sign(ac)H_{m}(z) & =-\frac{D+Ct_{1}}{(t_{1}-t_{2})(1-zt_{1})t_{1}^{m+1}}-\frac{D+Ct_{2}}{(t_{2}-t_{1})(1-zt_{2})t_{2}^{m+1}}+\frac{z^{m+1}(C+Dz)}{(1-zt_{1})(1-zt_{2})}\nonumber \\
 & =\frac{-t_{2}^{m+1}(1-zt_{2})(Ct_{1}+D)-t_{1}^{m+1}(zt_{1}-1)(Ct_{2}+D)+\sign^{m+1}(ac)z^{m+1}(t_{1}-t_{2})(C+Dz)}{\sign^{m+1}(ac)(t_{1}-t_{2})(1-zt_{1})(1-zt_{2})}.\label{eq:H_mform}
\end{align}

\begin{rem}
\label{rem:limitzeros}As a consequence of \cite[Theorem 1.5]{sokal},
if $|t_{2}|>|t_{1}|$, then the major portion of the limiting curve
(those satisfying condition (b) of this theorem), where the zeros
of the numerator of \eqref{eq:H_mform} approach as $m\rightarrow\infty$,
is the circle radius $|t_{2}|$. Thus we deduce from \eqref{eq:HmPm*}
and \eqref{eq:alphat2} that the majority of zeros of $P_{m}(z)$
approach the circle radius $|c|/|a\alpha|$. 
\end{rem}

\section{The location of zeros}

In this section, we will prove Theorems \ref{thm:maintheorem1}, \ref{thm:maintheorem2},
and \ref{thm:maintheorem3}. For this whole section, let $m\in\mathbb{N}$
be large.

\subsection{The case $ac<0$}

In the case $ac<0$, we have $t_{1}t_{2}=-1$ . The fact that $t_{1}+t_{2}=|b/\sqrt{ac}|>0$
and $|t_{2}|\ge|t_{1}|$ implies $-1<t_{1}<0$ and $t_{2}>1$. The
identity \eqref{eq:H_mform} becomes 
\begin{equation}
(-1)^{m}H_{m}(z)=\frac{-t_{2}^{m+1}(1-zt_{2})(Ct_{1}+D)-t_{1}^{m+1}(zt_{1}-1)(Ct_{2}+D)+(-1)^{m+1}z^{m+1}(t_{1}-t_{2})(C+Dz)}{(t_{1}-t_{2})(1-zt_{1})(1-zt_{2})}.\label{eq:H_mformacneg}
\end{equation}
To study the zeros of $H_{m}(z)$, we will apply Rouche's Theorem
to the numerator of the right side of \eqref{eq:H_mform}. The lemma
below plays an important role in the application of this theorem. 
\begin{prop}
\label{prop:Roucheineq}For sufficiently small $\epsilon>0$, if $|z|=t_{2}+\epsilon$
and $CD\ge0$, then 
\begin{equation}
\left|z^{m+1}(t_{1}-t_{2})(C+Dz)\right|>\left|t_{2}^{m+1}(1-zt_{2})(Ct_{1}+D)+t_{1}^{m+1}(zt_{1}-1)(Ct_{2}+D)\right|.\label{eq:Roucheineqacneg}
\end{equation}
\end{prop}

\begin{proof}
In the case $Ct_{1}+D=0$, the inequality for large $m$ follows immediately
from the fact that $|t_{2}|>|t_{1}|$. We consider $Ct_{1}+D\ne0$.
Since $t_{1}t_{2}=-1$, the left side of \eqref{eq:Roucheineqacneg}
is 
\begin{align*}
\left|z^{m+1}(t_{1}-t_{2})(C+Dz)\right| & =|z|^{m+1}(1/t_{2}+t_{2})|C+Dz|\\
 & \ge|z|^{m+1}(1/t_{2}+t_{2})||C|-|Dz||\\
\end{align*}
With $|z|=t_{2}+\epsilon$, we expand the last expression as a series
in $\epsilon$ and conclude it is at least 
\begin{equation}
t_{2}^{m}(t_{2}^{2}+1)||C|-|D|t_{2}|+(t_{2}^{2}+1)(||C|-|D|t_{2}|t_{2}^{m-1}(m+1)-|D|t_{2}^{m})\epsilon+\mathcal{O}(\epsilon^{2}).\label{eq:lefttermacneg}
\end{equation}
In the main term of the expression above, the fact that $CD\ge0$,
$t_{2}>0$, and $t_{1}t_{2}=-1$ imply 
\[
||C|-|D|t_{2}|=\frac{|Ct_{1}+D|}{t_{1}}\ne0.
\]

On the other hand by collecting the coefficients of $z$, the right
side of \eqref{eq:Roucheineqacneg} is 
\begin{align}
 & \left|t_{2}^{m+1}(1-zt_{2})(Ct_{1}+D)+t_{1}^{m+1}(zt_{1}-1)(Ct_{2}+D)\right|\nonumber \\
= & \left|z(-t_{2}^{m+2}(Ct_{1}+D)+t_{1}^{m+2}(Ct_{2}+D))+t_{2}^{m+1}(Ct_{1}+D)-t_{1}^{m+1}(Ct_{2}+D)\right|\nonumber \\
\le & (t_{2}+\epsilon)\left|t_{2}^{m+2}(Ct_{1}+D)-t_{1}^{m+2}(Ct_{2}+D)\right|+\left|t_{2}^{m+1}(Ct_{1}+D)-t_{1}^{m+1}(Ct_{2}+D)\right|\label{eq:sumabs}\\
\nonumber 
\end{align}
We note that for large $m$ 
\[
|t_{2}^{m+2}(Ct_{1}+D)|\ge|t_{1}^{m+2}(Ct_{2}+D)|
\]
and thus 
\[
\left|t_{2}^{m+2}(Ct_{1}+D)-t_{1}^{m+2}(Ct_{2}+D)\right|=t_{2}^{m+2}|Ct_{1}+D|-\delta|t_{1}^{m+2}||Ct_{2}+D|
\]
where $\delta=\sign((-1)^{m+2}(Ct_{1}+D)(Ct_{2}+D)$. From a similar
identity 
\[
\left|t_{2}^{m+1}(Ct_{1}+D)-t_{1}^{m+1}(Ct_{2}+D)\right|=t_{2}^{m+1}|Ct_{1}+D|+\delta|t_{1}^{m+1}||Ct_{2}+D|,
\]
\eqref{eq:sumabs} becomes 
\begin{align*}
 & (t_{2}+\epsilon)\left(t_{2}^{m+2}|Ct_{1}+D|-\delta|t_{1}^{m+2}||Ct_{2}+D|\right)+t_{2}^{m+1}|Ct_{1}+D|+\delta|t_{1}^{m+1}||Ct_{2}+D|\\
= & t_{2}^{m+3}|Ct_{1}+D|+t_{2}^{m+1}|Ct_{1}+D|+\left(t_{2}^{m+2}|Ct_{1}+D|-\delta|t_{1}^{m+2}||Ct_{2}+D|\right)\epsilon\\
= & t_{2}^{m+2}|C-Dt_{2}|+t_{2}^{m}|C-Dt_{2}|+\left(t_{2}^{m+1}|C-Dt_{2}|-\delta|t_{1}^{m+1}||C-Dt_{2}|\right)\epsilon.
\end{align*}
The inequality $CDt_{2}\ge0$ implies $|C-Dt_{2}|=||C|-|D|t_{2}|$.
We compare the last expression above with \eqref{eq:lefttermacneg}
and complete the lemma using the fact that 
\[
(t_{2}^{2}+1)(||C|-|D|t_{2}|t_{2}^{m-1}(m+1)-|D|t_{2}^{m})>t_{2}^{m+1}|C-Dt_{2}|-\delta|t_{1}^{m+1}||C-Dt_{2}|
\]
for large $m$. 
\end{proof}
From Proposition \ref{prop:Roucheineq} and the Rouche's Theorem,
we conclude that for small $\epsilon$, the number of zeros of the
polynomial 
\begin{equation}
-t_{2}^{m+1}(1-zt_{2})(Ct_{1}+D)-t_{1}^{m+1}(zt_{1}-1)(Ct_{2}+D)+(-1)^{m+1}z^{m+1}(t_{1}-t_{2})(C+Dz)\label{eq:rouchepolyacneg}
\end{equation}
inside the ball $|z|<t_{2}+\epsilon$ is 
\[
\begin{cases}
m+1 & \text{ if }|C|>|D|t_{2}\\
m+2 & \text{ if }|C|\le|D|t_{2}.
\end{cases}
\]
We can check that $-t_{2}$ and $-t_{1}$ are two zeros of \eqref{eq:rouchepolyacneg}.
We deduce from \eqref{eq:H_mformacneg} that $H_{m}(z)$ has 
\[
\begin{cases}
m-1 & \text{ if }|C|>|D|t_{2}\\
m & \text{ if }|C|\le|D|t_{2}
\end{cases}
\]
zeros inside the closed ball $|z|\le t_{2}$. From 
\begin{align*}
C & =-(a_{1}+ba_{0}/c)\sign(ab)\sqrt{|c/a|},\\
D & =a_{0},
\end{align*}
and \eqref{eq:alphat2}, the inequalities $|C|\le|D|t_{2}$ and $CD\ge0$
are equivalent to 
\[
|a_{1}c+ba_{0}|\le|a\alpha a_{0}|
\]
and 
\[
a_{0}b(a_{1}c+ba_{0})\ge0
\]
respectively. From \eqref{eq:HmPm*} and \eqref{eq:alphat2}, we conclude
that $z$ is a zero of $H_{m}(z)$ inside the closed ball $|z|\le t_{2}$
if and only if $-\sign(ab)\sqrt{|c/a|}z$ is a zero of $P_{m}^{*}(z)$
inside the closed ball radius $\alpha\sqrt{|a/c|}$. Theorem \pageref{thm:maintheorem1}
follows from the fact that $z\ne0$ is a zero of $P_{m}^{*}(z)$ if
and only if $1/z$ is a zero of $P_{m}(z)$.

\subsection{The case $ac>0$ and $b^{2}-4ac>0$}

In the case $ac>0$, we have $t_{1}t_{2}=1$ and $0<t_{1}<1<t_{2}$.
The identity \eqref{eq:H_mform} becomes 
\begin{equation}
H_{m}(z)=\frac{-t_{2}^{m+1}(1-zt_{2})(Ct_{1}+D)-t_{1}^{m+1}(zt_{1}-1)(Ct_{2}+D)+z^{m+1}(t_{1}-t_{2})(C+Dz)}{(t_{1}-t_{2})(1-zt_{1})(1-zt_{2})}.\label{eq:H_mformacpos}
\end{equation}

\begin{prop}
\label{prop:RoucheineqCDpos}For sufficiently small $\epsilon>0$,
if $|z|=t_{2}-\epsilon$ and $CD\ge0$, then 
\[
\left|z^{m+1}(t_{1}-t_{2})(C+Dz)\right|<\left|t_{2}^{m+1}(1-zt_{2})(Ct_{1}+D)+t_{1}^{m+1}(zt_{1}-1)(Ct_{2}+D)\right|.
\]
\end{prop}

\begin{proof}
The proof is very similar to that of Proposition \ref{prop:Roucheineq}.
To avoid repetition, we only provide key steps. The left side of the
inequality in consideration is at most 
\[
t_{2}^{m}(t_{2}^{2}-1)(|C|+|D|t_{2})+(t_{2}^{2}-1)(-(|C|+|D|t_{2})t_{2}^{m-1}(m+1)-|D|t_{2}^{m})\epsilon+\mathcal{O}(\epsilon^{2})
\]
while, after collecting the coefficients of $z$, the right side is
at least 
\[
(t_{2}-\epsilon)\left|t_{2}^{m+2}(Ct_{1}+D)-t_{1}^{m+2}(Ct_{2}+D)\right|-\left|t_{2}^{m+1}(Ct_{1}+D)-t_{1}^{m+1}(Ct_{2}+D)\right|.
\]
With $\delta=\sign((Ct_{1}+D)(Ct_{2}+D)$, the last expression becomes
\[
(t_{2}-\epsilon)\left(t_{2}^{m+2}|Ct_{1}+D|-\delta t_{1}^{m+2}|Ct_{2}+D|\right)-t_{2}^{m+1}|Ct_{1}+D|+\delta t_{1}^{m+1}|Ct_{2}+D|
\]
which is 
\[
t_{2}^{m+2}|C+Dt_{2}|-t_{2}^{m}|C+Dt_{2}|+\left(-t_{2}^{m+1}|C+Dt_{2}|+\delta t_{1}^{m+1}|C+Dt_{2}|\right)\epsilon.
\]
The lemma follows from the fact that $CD\ge0$. 
\end{proof}

We now prove Theorem \ref{thm:maintheorem2}. From the fact that $a_{0}b(a_{1}c+ba_{0})<0$
is equivalent to $CD>0$ and Proposition \ref{prop:RoucheineqCDpos},
Rouche's Theorem implies that the number of zeros of the numerator
of \eqref{eq:H_mformacpos} inside $|z|<t_{2}-\epsilon$ is the same
as that of 
\[
t_{2}^{m+1}(1-zt_{2})(Ct_{1}+D)+t_{1}^{m+1}(zt_{1}-1)(Ct_{2}+D).
\]
This linear function has a zero 
\[
z=\frac{t_{2}^{m+1}(Ct_{1}+D)-t_{1}^{m+1}(Ct_{2}+D)}{t_{2}^{m+2}(Ct_{1}+D)-t_{1}^{m+1}(Ct_{2}+D)}
\]
which approaches $1/t_{2}$ as $m\rightarrow\infty$. Thus this zero
lies inside the circle $|z|<t_{2}-\epsilon$ for small $\epsilon$
and large $m$. Since $t_{1}$ is a zero of the numerator of \eqref{eq:H_mformacpos},
we let $\epsilon\rightarrow0$ and conclude that all zeros of $H_{m}(z)$
lie outside the circle $|z|\le t_{2}$. Theorem \ref{thm:maintheorem2}
follows from \eqref{eq:HmPm*}.

We now turn our attention to the proof of Theorem \ref{thm:maintheorem3}.
The condition $a_{0}b(a_{1}c+ba_{0})>0$ is equivalent to $CD<0$.
We note from \eqref{eq:H_mformacpos} that it suffices to consider
$D>0$ and $C<0$. The condition $|a_{0}c|>|(a_{1}c+ba_{0})\alpha|$
in Theorem \ref{thm:maintheorem3} is equivalent to $|D|>|Ct_{2}|$.
Together with the fact that $D>0$, $C<0$, and $t_{2}>0$, we conclude
$D+Ct_{2}>0$. With the same application of Rouche's Theorem, Theorem
\ref{thm:maintheorem3} follows directly from the proposition below. 
\begin{prop}
\label{prop:RoucheineqCDneg}For sufficiently small $\epsilon>0$,
if $D>0$, $C<0$, $D+Ct_{2}>0$, and $|z|=t_{2}-\epsilon$ , then
\begin{equation}
\left|z^{m+1}(t_{1}-t_{2})(C+Dz)\right|<\left|t_{2}^{m+1}(1-zt_{2})(Ct_{1}+D)+t_{1}^{m+1}(zt_{1}-1)(Ct_{2}+D)\right|.\label{eq:CDnegineq}
\end{equation}
\end{prop}

\begin{proof}
From $t_{1}t_{2}=1$, the right side of \eqref{eq:CDnegineq} is the
modulus of 
\[
t_{2}^{m+1}(t_{1}-z)(C+Dt_{2})+t_{1}^{m+1}(zt_{1}-1)(Ct_{2}+D).
\]
We next compare the modulus of the first term in the expression above
with the left side of \eqref{eq:CDnegineq}. We note that if $z=x+iy$,
then from $x^{2}+y^{2}=(t_{2}-\epsilon)^{2}$, we conclude that 
\begin{align*}
 & |(t_{1}-z)(C+Dt_{2})|^{2}-|(t_{1}-t_{2})(C+Dz)|^{2}\\
= & (C+Dt_{2})^{2}(t_{1}^{2}-2t_{1}x+(t_{2}-\epsilon)^{2})-(t_{1}-t_{2})^{2}(C^{2}+2CDx+D^{2}(t_{2}-\epsilon)^{2})\\
= & 2(C+Dt_{1})(Ct_{1}+Dt_{2}^{2})(t_{2}-x)-2t_{2}(C+Dt_{1})(C-Dt_{1}+2Dt_{2})\epsilon+\mathcal{O}(\epsilon^{2}).
\end{align*}
Since $D+Ct_{2}>0$ and $t_{2}>1>t_{1}$ the inequality \eqref{eq:CDnegineq}
holds for $\epsilon$ is small and $m$ is large if $t_{2}-x>\eta$
for some small $\eta>0$ independent of $m,\epsilon$. We now consider
$z=t_{2}+\delta$ for small $\delta=\delta(m,\epsilon)\in\mathbb{C}$.
We have 
\[
(t_{2}-\epsilon)^{2}=|z|^{2}=t_{2}^{2}+2t_{2}\Re\delta+|\delta|^{2}
\]
from which we deduce 
\begin{equation}
|\delta|^{2}=-2t_{2}\Re\delta-2t_{2}\epsilon+\mathcal{O}(\epsilon^{2}).\label{eq:deltaepsilon}
\end{equation}
We note that with $t_{1}t_{2}=1$ and $z=t_{2}+\delta$, the expression
inside the right side of \eqref{eq:CDnegineq} is 
\begin{align*}
 & t_{2}^{m+1}(t_{1}-z)(C+Dt_{2})+t_{1}^{m+1}(zt_{1}-1)(Ct_{2}+D)\\
= & t_{2}^{m+1}(t_{1}-t_{2})(C+Dt_{2})+\left(-t_{2}^{m+1}(C+Dt_{2})+t_{1}^{m+2}(Ct_{2}+D)\right)\delta.
\end{align*}
From the fact that $t_{1}t_{2}=1$, the square of modulus of the last
expression is 
\[
t_{2}^{2m+2}(t_{2}-t_{1})^{2}(C+Dt_{2})^{2}+(C+Dt_{2})^{2}t_{2}^{2m+2}(-2(t_{1}-t_{2})\Re\delta+|\delta|^{2})+\mathcal{O}\left(|\Re\delta|+|\delta|^{2}\right)
\]
which (from \eqref{eq:deltaepsilon}) is the same as 
\begin{equation}
t_{2}^{2m+2}(t_{2}-t_{1})^{2}(C+Dt_{2})^{2}+(C+Dt_{2})^{2}t_{2}^{2m+2}(-2t_{2}\epsilon-2t_{1}\Re\delta)+\mathcal{O}\left(|\Re\delta|+|\delta|^{2}+t_{2}^{2m}\epsilon^{2}\right).\label{eq:rightCDneg}
\end{equation}

We now consider the left side of \eqref{eq:CDnegineq}. We note that
\begin{align*}
|z|^{m+1} & =(t_{2}-\epsilon)^{m+1}\\
 & \le t_{2}^{m+1}\exp\left(-\frac{m+1}{t_{2}}\epsilon\right).
\end{align*}
Together with 
\begin{align*}
|C+Dz|^{2} & =(C+Dt_{2}+D\Re\delta)^{2}+D^{2}(\Im\delta)^{2}\\
 & =(C+Dt_{2})^{2}+2(C+Dt_{2})D\Re\delta+D^{2}|\delta|^{2},
\end{align*}
we conclude that the square of the left side of \eqref{eq:CDnegineq}
is at most 
\[
\exp\left(-2\frac{m+1}{t_{2}}\epsilon\right)t_{2}^{2m+2}(t_{2}-t_{1})^{2}\left((C+Dt_{2})^{2}+2(C+Dt_{2})D\Re\delta+D^{2}|\delta|^{2}\right).
\]
Using \eqref{eq:deltaepsilon}, we rewrite the express above as 
\begin{equation}
\exp\left(-2\frac{m+1}{t_{2}}\epsilon\right)t_{2}^{2m+2}(t_{2}-t_{1})^{2}\left((C+Dt_{2})^{2}+2CD\Re\delta-2D^{2}t_{2}\epsilon+\mathcal{O}(\epsilon^{2})\right).\label{eq:leftCDneg}
\end{equation}

If $m\epsilon>\eta_{1}$ for some small $\eta_{1}$, independent of
$m$ and $\epsilon$, then 
\[
\exp\left(-2\frac{m+1}{t_{2}}\epsilon\right)<1-\eta
\]
for some fixed (independent of $m$ and $\epsilon$) $\eta>0$. We
compare the main terms of \eqref{eq:rightCDneg} and \eqref{eq:leftCDneg}
and conclude the lemma. On the other hand, for sufficiently small
$m\epsilon$ \footnote{$m\epsilon$ is small enough so that $\mathcal{O}\left(m^{2}\epsilon^{2}\right)$
is dominated by the two main terms in \eqref{eq:expexpansion}. The
big-Oh constant is independent of $m,\epsilon$.}, we have 
\begin{equation}
\exp\left(-2\frac{m+1}{t_{2}}\epsilon\right)=1-2\frac{m+1}{t_{2}}\epsilon+\mathcal{O}\left(m^{2}\epsilon^{2}\right)\label{eq:expexpansion}
\end{equation}
and thus \eqref{eq:leftCDneg} is 
\begin{align*}
 & t_{2}^{2m+2}(t_{2}-t_{1})^{2}(C+Dt_{2})^{2}+t_{2}^{2m+2}(t_{2}-t_{1})^{2}\left(2CD\Re\delta-2D^{2}t_{2}\epsilon\right)\\
 & -2\frac{m+1}{t_{2}}t_{2}^{2m+2}(t_{2}-t_{1})^{2}(C+Dt_{2})^{2}\epsilon+\mathcal{O}\left(t_{2}^{2m}(m|\Re\delta|\epsilon+m^{2}\epsilon^{2})\right).
\end{align*}
If we subtract the expression above from \eqref{eq:rightCDneg}, we
get 
\begin{align}
 & (C+Dt_{2})^{2}t_{2}^{2m+2}(-2t_{2}\epsilon-2t_{1}\Re\delta)-t_{2}^{2m+2}(t_{2}-t_{1})^{2}\left(2CD\Re\delta-2D^{2}t_{2}\epsilon\right)\nonumber \\
 & +2\frac{m+1}{t_{2}}t_{2}^{2m+2}(t_{2}-t_{1})^{2}(C+Dt_{2})^{2}\epsilon+\mathcal{O}\left(t_{2}^{2m}(m|\Re\delta|\epsilon+m^{2}\epsilon^{2})+|\Re\delta|+|\delta|^{2}\right).\label{eq:CDnegdif}
\end{align}
If $\epsilon>\eta\Re\delta$ for some small $\eta$ independent of
$m,\epsilon,\delta$, then the third term of \eqref{eq:CDnegdif}
dominates all other terms when $m$ is large and thus \eqref{eq:CDnegdif}
is positive for large $m$. On the other hand when $\epsilon/\Re\delta$
is sufficiently small, the main term of \eqref{eq:CDnegdif} is 
\begin{align*}
 & -2(C+Dt_{2})^{2}t_{2}^{2m+2}t_{1}\Re\delta-2CDt_{2}^{2m+2}(t_{2}-t_{1})^{2}\Re\delta+2\frac{m+1}{t_{2}}t_{2}^{2m+2}(t_{2}-t_{1})^{2}(C+Dt_{2})^{2}\epsilon\\
= & -2t_{2}^{2m+2}(C+Dt_{1})(Ct_{1}+Dt_{2}^{2})\Re\delta+2\frac{m+1}{t_{2}}t_{2}^{2m+2}(t_{2}-t_{1})^{2}(C+Dt_{2})^{2}\epsilon>0.
\end{align*}
We conclude the proof of the lemma. 
\end{proof}

\end{document}